\documentclass[11pt, a4paper]{article}

\usepackage[utf8]{inputenc}
\usepackage{amsmath, amsthm, amsfonts, amssymb}

\usepackage{indentfirst} 

\usepackage{authblk}
\usepackage{lipsum}

\usepackage{mleftright}

\providecommand{\keywords}[1]
{
  \small	
  \textbf{\textit{Keywords---}} #1
}

\newtheorem{lemma}{Lemma}[section]
\newtheorem{theorem}[lemma]{Theorem}
\newtheorem{proposition}[lemma]{Proposition}
\newtheorem{corollary}[lemma]{Corollary}

\newtheorem{result}{Theorem}

\newtheorem{corollaryresult}[result]{Corollary}

\theoremstyle{definition}

\newcommand{\abs}[1]{\ensuremath{\left| #1 \right|}}
\newcommand{\op}{\operatorname}
\newcommand{\ce}[2]{\operatorname{C}_{#1}(#2)}
\newcommand{\no}[2]{\operatorname{N}_{#1}(#2)}
\newcommand{\ze}[1]{\operatorname{Z}(#1)}

\newcommand{\syl}[2]{\op{Syl}_{#1}\left(#2\right)}

\newcommand{\groupgen}[1]{\langle #1 \rangle}

\newcommand{\irr}{\operatorname{Irr}}
\newcommand{\lin}{\operatorname{Lin}}

\newcommand{\gal}{\operatorname{Gal}}
\newcommand{\cycl}[1]{\mathbb{Q}_{#1}}
\newcommand{\smid}{\! \mid \!}

\def\PSL{\mathrm{PSL}}
\def\PSU{\mathrm{PSU}}
\def\PGL{\mathrm{PGL}}
\def\SL{\mathrm{SL}}
\def\SU{\mathrm{SU}}
\def\GL{\mathrm{GL}}
\def\Sz{\mathrm{Sz}}
\def\Aut{\mathrm{Aut}}

\def\G{\mathcal{G}}
\def\H{\mathcal{H}}

\renewcommand{\phi}{\varphi}
\renewcommand{\epsilon}{\varepsilon}

\title{Sylow normalizers and irreducible characters with small cyclotomic field of values}

\date{\today}

\author{Nicola Grittini\footnote{nicola.grittini@unicatt.it} }
\author{Marco Antonio Pellegrini\footnote{marcoantonio.pellegrini@unicatt.it} }

\affil{Università Cattolica del Sacro Cuore\\
Via della Garzetta, 48 -- 25133 -- Brescia (Italy)}

\begin{document}

\maketitle

\begin{abstract}
\noindent
In this paper, we prove the existence of a relation between the prime divisors of the order of a Sylow normalizer and the degree of characters having values in some small cyclotomic fields. This relation is stronger when the group is solvable.

\end{abstract}

\section{Introduction}

In character theory, there had been surprisingly few attempts to find a way to determine from the character table which prime numbers divide the order of a Sylow normalizer. In \cite{Isaacs-Navarro:Character_table_Sylow_normalizer}, it is proved that it is possible to do it for solvable groups. In particular, it is proved that it is possible to tell from the character table of a group $G$ whether a prime number $r$ divides the order of the normalizer of a Sylow $p$-subgroup, for some other prime number $p$, if we assume that the group $G$ is $r$-solvable. If $G$ is $p$-solvable, on the other hand, \cite[Theorem~C]{Guralnick-Navarro-Tiep:Finite_groups_odd_normalizer} provides a sufficient condition, involving character values, for the normalizer of a Sylow $p$-subgroup to be of even order.

In this paper, we will see that the character table can provide some informations on which prime numbers divide the order of a Sylow normalizer, also without assumptions on group solvability.

This paper, however, is motivated also by another, apparently unrelated problem. In \cite{Navarro-Tiep:Rational_irreducible_characters}, answering to a question first posed by R. Gow, Navarro and Tiep proved that every finite group of even order has a rational-valued character of odd degree. In a subsequent paper \cite{Navarro-Tiep:Character_cyclotomic_field} (which appeared first in literature), they expanded their previous result by proving that, for any prime number $p$ and any finite group $G$ of order divisible by $p$, there exists a character in $\irr(G)$ of degree not divisible by $p$ and having values in a $p$-cyclotomic extension of $\mathbb{Q}$, i.e., in the field $\mathbb{Q}(\zeta)$ where $\zeta$ is some primitive $p$th-root of unity.

A more recent contribution to the topic is \cite{Giannelli-Hung-ShaefferFry-Vallejo:Characters_of_pi_degree_and_small}, where it is proved that, with very few exceptions, finite groups always have a non-principal irreducible character whose degree is not divisible by \emph{two} given primes and whose field of values is contained in some small cyclotomic extension of $\mathbb{Q}$.

In this paper, we will study another possible generalization of the results in \cite{Navarro-Tiep:Character_cyclotomic_field}. Given \emph{two} prime numbers $p$ and $r$, we will see under which conditions a group $G$ has an irreducible character of degree not divisible by $p$ having values in some $r$-cyclotomic extension of $\mathbb{Q}$. We will see that this always happens when the normalizer of a Sylow $p$-subgroup of the group $G$ satisfies some conditions which are automatically satisfied if $p=2$ or $p=r$, as we would expect from the results in \cite{Navarro-Tiep:Character_cyclotomic_field}.

First, we need to establish some notation. Given a positive integer $n$, we write $\cycl{n}$ to indicate the $n$-cyclotomic extension of the field $\mathbb{Q}$, i.e., the field $\mathbb{Q}(\zeta)$ where $\zeta$ is some primitive $n$th-root of unity. This is a standard notation in character theory but, unfortunately, if $n=p$ the same notation is just as commonly used to denote the field of $p$-adic numbers. In this paper, $\cycl{p}$ will always denote the $p$-cyclotomic extension of $\mathbb{Q}$.

We are interested in the irreducible characters of a group $G$ having values in $\cycl{r}$, for some prime number $r$, and with a degree not divisible by a prime number $p$. Sometimes, we may write this set as
$$ \irr_{p',\cycl{r}}(G) = \Big\{ \chi \in \irr(G) \; \Big| \; p \nmid \chi(1) \mbox{ and } \mathbb{Q}(\chi) \subseteq \cycl{r} \Big\}. $$

The aforementioned result of Navarro and Tiep, \cite[Theorem~A]{Navarro-Tiep:Rational_irreducible_characters}, asserts that the set $\irr_{p',\cycl{p}}(G)$ contains a non-principal character whenever the order of $G$ is ether even or divisible by $p$. It is possible, however, that $\irr_{p',\cycl{r}}(G) = \{ 1_G \}$ if $p \neq r$, even when the prime numbers $2$, $p$ and $r$ all divide $\abs{G}$. An example of this fact is the group $G = \op{P\Sigma L}(2,27)$, i.e., the group formed by taking the semidirect product of $\op{PSL}(2,27)$ acted on by the cyclic group generated by a field automorphism of order $3$ (notice that this counterexample also appears in \cite{Isaacs-Malle-Navarro:Real_characters}). In this situation, all the irreducible characters in $\irr_{3'}(G)$ are non-rational with values in $\cycl{3}$; therefore, $\irr_{3',\cycl{r}}(G) = \{ 1_G \}$ for any $r \neq 3$.

Thus, it does make sense to find some conditions that allow the set $\irr_{p',\cycl{r}}(G)$ to contain a non-principal character. This is the aim of our main theorem.

\begin{result}
\label{result:general}
Let $G$ be a finite group and let $p$ and $r$ be two (possibly equal) prime numbers. Let $P \in \syl{p}{G}$; if either $2$ or $r$ divide $\abs{\no{G}{P}}$, there exists a non-principal character $\chi \in \irr_{p'}(G)$ having values in $\cycl{r}$.
\end{result}

We may notice that, if we take $r=p$, then Theorem~\ref{result:general} is exactly \cite[Theorem~A]{Navarro-Tiep:Character_cyclotomic_field} (it is actually one direction of \cite[Theorem~B]{Navarro-Tiep:Character_cyclotomic_field}). Moreover, if $p=2$, then the theorem follows directly from \cite[Theorem~B]{Navarro-Tiep:Rational_irreducible_characters}. In fact, we may say that Theorem~\ref{result:general} extends \cite[Theorem~A]{Navarro-Tiep:Character_cyclotomic_field} for an odd prime number $p \neq r$.

On the other hand, if we take $r=2$, we can have some information about the existence of rational-valued irreducible characters of $p'$-degree. This information can integrate the classification we have from \cite[Theorem~C]{Navarro-Tiep:Character_cyclotomic_field}.

\begin{corollaryresult}
Let $G$ be a finite group and let $P \in \syl{p}{G}$ for some prime number $p$. If $\abs{\no{G}{P}}$ is even, then $G$ has a non-principal rational-valued irreducible character of degree not divisible by $p$.
\end{corollaryresult}

The conditions presented in \cite[Theorem~B]{Navarro-Tiep:Character_cyclotomic_field} are both necessary and sufficient for the existence of a non-principal character in $\irr_{p',\cycl{p}}(G)$. This is not true in our case. In fact, it is possible that a group $G$ has a non-principal character in $\irr_{p',\cycl{r}}(G)$ even if $(\abs{\no{G}{P}},2r)=1$ for some Sylow $p$-subgroup $P \leq G$. Examples of this fact are the alternating groups $\op{A}_7$ and $\op{A}_8$ for $p=7$ and $r = 5$, the group $\op{PSL}(2,19)$ for $p=19$ and $r=5$, and the group $\op{PSL}(2,27)$ for $p=3$ and $r=7$. However, the existence of a non-principal character in $\irr_{p',\cycl{r}}(G)$ can still tell us something on the group structure. In particular, we have a generalization of \cite[Theorem~C]{Guralnick-Navarro-Tiep:Finite_groups_odd_normalizer}.

\begin{result}
\label{result:solvable_subgroup}
Let $G$ be a finite group and let $p$ and $r$ be two primes, let $P \in \syl{p}{G}$ and assume that $(\abs{\no{G}{P}},2r)=1$. If $\chi \in \irr_{p'}(G)$ has values in $\cycl{r}$, then $N \leq \ker(\chi)$ for any $p$-solvable $N \lhd G$.
\end{result}

In particular, if in Theorem~\ref{result:solvable_subgroup} we take $N=G$, we have that for $p$-solvable groups the condition in Theorem~\ref{result:general} on the order of the Sylow normalizer is not only sufficient but also necessary. In particular, this generalizes \cite[Corollary~A]{Turull:Odd_character_correspondences}.

\begin{corollaryresult}
\label{corollaryresult:solvable}
Let $G$ be a $p$-solvable group for some prime number $p$ and let $P \in \syl{p}{G}$. Let $r$ be a prime number, then there exists a non-principal character $\chi \in \irr_{p'}(G)$ having values in $\cycl{r}$ if and only if either $2$ or $r$ divide $\abs{\no{G}{P}}$.
\end{corollaryresult}

For solvable groups, however, there is more we can say. If $P$ is the Sylow $p$-subgroup of a finite group $G$, we may write $\irr({1_P}^G)$ to denote the set of irreducible constituents of ${1_P}^G$. In \cite{Malle-Navarro:Characterizing_normal_sylow}, it is proved that $\irr_{p'}({1_P}^G)=\irr({1_P}^G)$ if and only if $P \lhd G$. On the other hand, if $p$ is odd, it is proved in \cite{Navarro-Tiep-Vallejo:McKay_natural_correspondence} that $\irr_{p'}({1_P}^G)=\{ 1_G \}$ if and only if $\no{G}{P}=P$.

In our situation, if the group $G$ is $p$-solvable, we can consider the set $\irr_{p',\cycl{r}}({1_P}^G)$ instead of $\irr_{p',\cycl{r}}(G)$ and still have the same results, with few exceptions.

\begin{result}
\label{result:solvable}
Let $G$ be a $p$-solvable group, for some odd prime number $p$ which is not a Fermat prime, and let $P \in \syl{p}{G}$. Let $r \neq p$ be a prime number, then there exists a non-principal character $\chi \in \irr_{p'}({1_P}^G)$ having values in $\cycl{r}$ if and only if either $2$ or $r$ divide $\abs{\no{G}{P}}$.
\end{result}

We recall that a number $p$ is a \emph{Fermat prime} if and only if it is a prime number and $p - 1$ is a power of $2$. There are currently only 5 known Fermat primes and it is believed that there are no more.

Theorem~\ref{result:solvable} fails if $p$ is a Fermat prime. A counterexample is given by the group $G=\op{SL}(2,3)$, for $p = 3$ and $r = 2$, since in this case $\irr_{3',\cycl{}}(G)$ contains a single non-principal character and it does not lie over the principal character of a Sylow $3$-subgroup of $G$.

We mention that it is possible that Theorem~\ref{result:solvable} holds also without assumptions on group solvability. However, our technique cannot be employed with groups which are not $p$-solvable.

\section{Primes dividing $\mathbf{\abs{\no{G}{P}}}$ in $\mathbf{p}$-solvable groups}
\label{section:solvable_case}

In this section we will prove the following result, which in fact is equivalent to Theorem~\ref{result:solvable_subgroup} and to one direction of Theorem~\ref{result:solvable}.

\begin{theorem}
\label{theorem:solvable_primes_dividing}
Let $p$ and $r$ be two prime numbers, let $G$ be a finite group, let $N$ be a $p$-solvable normal subgroup of $G$ and let $P$ be a Sylow $p$-subgroup of $G$. Let $\chi \in \irr_{p'}(G)$ having values in $\cycl{r}$, then either $N \leq \ker(\chi)$ or $(\abs{\no{G}{P}},2r)>1$.
\end{theorem}

Clearly, if in Theorem~\ref{theorem:solvable_primes_dividing} we take $N=G$ and $\chi \neq 1_G$, we have one direction of Corollary~\ref{corollaryresult:solvable} and of Theorem~\ref{result:solvable}.

In order to prove Theorem~\ref{theorem:solvable_primes_dividing} we need some preliminary result, the first of them being the following lemma from \cite{Navarro-Tiep-Vallejo:McKay_natural_correspondence}.

\begin{lemma}[{\cite[Lemma~2.1]{Navarro-Tiep-Vallejo:McKay_natural_correspondence}}]
\label{lemma:conjugate}
Let $G$ be a finite group, let $N \lhd G$ and let $P \in \syl{p}{G}$ for some prime number $p$. Let $\phi_1, \phi_2 \in \irr(N)$ be $G$-conjugate and suppose they are both $P$-invariant. Then, $\phi_2 = {\phi_1}^g$ for some $g \in \no{G}{P}$.
\end{lemma}

We now cite another result, which we need in order to deal with non-abelian chief factors of $p'$-order. The following result is cited in \cite{Isaacs-Malle-Navarro:Real_characters} and it is a consequence of the classification of finite simple groups. Of course, we could avoid to appeal to the classification if we strengthen the hypothesis of Theorem~\ref{result:solvable_subgroup} and of Theorem~\ref{result:solvable} and we assume, respectively, $N$ and $G$ to be solvable.

\begin{theorem}[{\cite[Theorem~3.4]{Isaacs-Malle-Navarro:Real_characters}}]
\label{theorem:coprime_action}
Let $P$ be a $p$-group that acts by automorphisms on a finite group $G$ of order not divisible by $p$, and suppose that $\abs{\ce{G}{P}}$ is odd. Then $G$ is solvable.
\end{theorem}

We now study what happens when two characters of a normal subgroup of $G$ are both $G$-conjugate and $\sigma$-conjugate, for some $\sigma \in \gal(\cycl{\abs{G}} \mid \mathbb{Q})$.

\begin{lemma}
\label{lemma:G-sigma_conjugate}
Let $G$ be a finite group, let $N \lhd G$ and let $\phi \in \irr(N)$. Let $\sigma \in \gal(\cycl{\abs{G}} \mid \mathbb{Q})$ such that $o(\sigma)=r^a$, for some prime number $r$, and suppose that $\phi^{\sigma} \neq \phi$ and $\phi^{\sigma} = \phi^g$ for some $g \in G$. Then, $r \mid o(g)$.
\end{lemma}

\begin{proof}
First, notice that the actions of $G$ and of $\sigma$ on $\irr(N)$ commute. It follows that, if $\phi^g=\phi^{\sigma}$, then ${\op{I}_G(\phi)}^g = \op{I}_G(\phi^g) = \op{I}_G(\phi^{\sigma}) = \op{I}_G(\phi)$ and, thus, $g$ normalizes $I=\op{I}_G(\phi)$. Moreover, for any integer $k$ we have that
$$ \phi^{\sigma^k} = (\phi^{\sigma})^{\sigma^{k-1}} = (\phi^g)^{\sigma^{k-1}} = (\phi^{\sigma^{k-1}})^g $$
and, working by induction on $\abs{k}$, we have that $\phi^{\sigma^k} = \phi^{g^k}$. In particular, for $h = o(\sigma)$ we have that $\phi^{g^h} = \phi^{\sigma^h} = \phi$ and, thus, $g^h \in I$. It follows that the order of the element $gI \in \groupgen{gI}/I$ divides $h=r^a$ and, since $g \notin I$, we have that $r \mid o(gI)$ and, thus, $r \mid o(g)$.
\end{proof}

Before we can finally conclude this section, we need to introduce a special field automorphism that we are going to use in the proof of Theorem~\ref{theorem:solvable_primes_dividing}. First, let us briefly recall some notation. Let $\pi$ be a set of prime numbers, then $\pi'$ is its complementary set, i.e., the set all prime numbers which are not elements $\pi$. Moreover, a complex number $z \in \mathbb{C}$ is said to be a $\pi$-root of unity if $z^n = 1$ for some positive integer $n$ such that all the prime divisors of $n$ are contained in $\pi$. Notice that 1 is the only element in $\mathbb{C}$ to be both a $\pi$-root and a $\pi'$-root of unity.

For a finite group $G$ and a set of prime numbers $\pi$, let $\tau \in \gal(\cycl{\abs{G}} \mid \mathbb{Q})$ such that $\tau$ acts like the identity on all $\pi$-roots of unity in $\cycl{\abs{G}}$ and like complex conjugation on all $\pi'$-roots of unity. The existence of such field automorphism is discussed in \cite{Isaacs:Characters_groups_odd_order} and, coherently with the notation used there, we may refer to $\tau$ as a \emph{magic field automorphism} for $G$ with respect to $\pi$.

We are now ready to prove Theorem~\ref{theorem:solvable_primes_dividing}.

\begin{proof}[Proof of Theorem~\ref{theorem:solvable_primes_dividing}]
At first, we may notice that we can assume that either $2$ or $r$ divide $\abs{G}$, since otherwise there would be no non-trivial irreducible character having values in $\cycl{r}$, and we can also assume that $r \neq p$, since the thesis is trivial otherwise.

Let $\chi \in \irr_{p'}(G)$ having values in $\cycl{r}$; if $N \leq \ker(\chi)$ we are done, thus, we may assume that $N \nleq \ker(\chi)$ and we will prove that either 2 or $r$ divide $\abs{\no{G}{P}}$. At this point, there is no loss to assume that $\chi$ is faithful and $N$ is a minimal normal subgroup of $G$.

If $N$ is a non-abelian $p'$-group, then by Theorem~\ref{theorem:coprime_action} we have that the order of $\no{N}{P} = \ce{N}{P}$ is even and we are done. Thus, we can assume that $N$ is abelian and, since $N \nleq \ker(\chi)$ and $p \nmid \chi(1)$, we have that there exists a $P$-invariant non-principal irreducible constituent $\lambda \in \irr(N)$ of $\chi_N$. Let now $\tau \in \gal(\cycl{\abs{G}} \mid \mathbb{Q})$ be a magic field automorphism for $G$ with respect to $\pi=\{r\}$ and notice that $\chi^{\tau}=\chi$, since $\mathbb{Q}(\chi) \leq \cycl{r}$.

Since $\lambda$ is linear, we have that $\lambda^{\tau} = \lambda$ if and only if $o(\lambda)$ is equal to either 2 or $r$, i.e., if and only if $N$ is either a $2$- or a $r$-group (remember that $N$ is an elementary abelian group, because of its minimality as a normal subgroup of $G$). Since $\lambda$ is a non-principal $P$-invariant irreducible character, by Glauberman correspondence we have that $\no{N}{P} = \ce{N}{P} > 1$ and it follows that, if $\lambda^{\tau} = \lambda$, then either $2$ or $r$ divide $\abs{\no{G}{P}}$ and we are done.

Suppose now that $\lambda^{\tau} \neq \lambda$. Since the actions of $\tau$ and of $P$ on $\irr(N)$ commute, we have that also $\lambda^{\tau}$ is $P$-invariant and, since $\chi$ is $\tau$-invariant, we have that both $\lambda$ and $\lambda^{\tau}$ are irreducible constituents of $\chi_N$ and, thus, they are $G$-conjugate. It follows from Lemma~\ref{lemma:conjugate} that $\lambda^{\tau} = \lambda^g$ for some $g \in \no{G}{P}$. Moreover, since $o(\tau)=2$, we have from Lemma~\ref{lemma:G-sigma_conjugate} that $2 \mid o(g)$ and it follows that $2 \mid \abs{\no{G}{P}}$ also in this case.
\end{proof}

\section{Isomorphism of character triples}

In this section, we will be dealing with the theory of central extensions and character triples isomorphisms. The notation and terminology we will use in this section are standard for the theory and, for this reason, we will omit some definitions. We refer to \cite[Chapter~11]{Isaacs} whenever the terms we use result unfamiliar to the reader.

Given a group $G$, Schur theorem says that we can always find a central extension $(\Gamma, \pi)$ of $G$ with the projective lifting property and such that the \emph{standard map} $\eta : \irr(\ker(\pi)) \mapsto M(G)$ is an isomorphism (see \cite[Theorem~11.17]{Isaacs}). Such group $\Gamma$ is called a \emph{Schur representation group} for $G$ and it is a group with the projective lifting property for $G$ with smallest possible order. 

Let $(G,N,\phi)$ be a character triple and let $(\Gamma, \pi)$ be a central extension of $G/N$ such that $\Gamma$ is a Schur representation group for $G/N$. If $A = \ker(\pi)$, we know from \cite[Theorem~11.28]{Isaacs} and \cite[Problem~11.13]{Isaacs} that $(G,N,\phi)$ and $(\Gamma,A,\lambda)$ are strongly isomorphic character triples, for some $\lambda \in \lin(A)$ which is uniquely determined by $\phi$ and by the standard map $\eta$. We may call $(\Gamma,A,\lambda)$ a \emph{canonically constructed} character triple isomorphic to $(G,N,\phi)$.

The two main theorems of this section will prove some properties concerning a canonically constructed character triple isomorphism, particularly for what concerns character field of values. Our results are sometimes similar to the ones presented by Turull in \cite{Turull:Clifford_theory}. In fact, both our hypothesis and our thesis are weaker then the ones of Turull, since we do no require that $\mathbb{Q}(\phi) = \mathbb{Q}(\lambda)$.

\begin{theorem}
\label{theorem:character_triple}
Let $(G,N,\phi)$ be a character triple, and let $(\Gamma,A,\lambda)$ be a canonically constructed character triple isomorphic to $(G,N,\phi)$. For each $\sigma \in \gal( \cycl{ \abs{G} } \mid \mathbb{Q} )$, if $\phi^{\sigma} = \phi$, then also $\lambda^{\sigma} = \lambda$.
\end{theorem}

\begin{proof}
We shall recall first some facts and definitions from \cite[Chapter~11]{Isaacs}.

\begin{itemize}
\item[a)] Let $\mathfrak{Y}$ be a representation of $N$ affording $\phi$, let $\mathfrak{X}$ be a projective representation associated with $\mathfrak{Y}$ according to \cite[Theorem~11.2]{Isaacs} and let $\alpha \in Z(G,\mathbb{C}^{\times})$ be the factor set associated with $\mathfrak{X}$.
\item[b)] Let $\beta \in Z(G/N,\mathbb{C}^{\times})$ such that $\beta(aN,bN)=\alpha(a,b)$, which is well defined, as proved in \cite[Theorem~11.7]{Isaacs}, and let $\overline{\beta}$ be its image in $H(G/N,\mathbb{C}^{\times})$.
\item[c)] Let $\alpha_1 \in Z(G/N,A)$ be the $A$-factor set associated with $\Gamma$ as in \cite[Lemma~11.9]{Isaacs} and, for $\gamma \in \irr(A)$, define $\gamma(\alpha_1) \in Z(G/N,\mathbb{C}^{\times})$ as $\gamma(\alpha_1)(g,h) = \gamma(\alpha_1(g,h))$, and define $\overline{\gamma(\alpha_1)}$ as the image of $\gamma(\alpha_1)$ in $H(G/N,\mathbb{C}^{\times})$. Let
$$ \eta: \irr(A) \mapsto H(G/N,\mathbb{C}^{\times}) $$
be the standard map, defined as $\eta(\gamma)=\overline{\gamma(\alpha_1)}$. Remember that, as a consequence of \cite[Theorem~11.17]{Isaacs}, the map $\eta$ is a group isomorphism.
\item[d)] As in the proof of \cite[Theorem~11.28]{Isaacs}, $\lambda$ is the element of $\irr(A)$ such that $\eta(\lambda)=\overline{\beta}^{-1}$.
\end{itemize}

Since $\eta$ is an isomorphism, in order to prove that $\lambda^{\sigma}=\lambda$ it is enough to show that $\eta(\lambda^{\sigma}) = \eta(\lambda)$. For this, we need to prove that
\begin{itemize}
\item[1)] $\eta(\lambda^{\sigma}) = (\eta(\lambda))^{\sigma}$ and
\item[2)] $\overline{\beta}^{\sigma} = \overline{\beta}$.
\end{itemize}

The proof of 1) follows easily from the fact that, by definition, $\lambda^{\sigma}(a) = (\lambda(a))^{\sigma}$ for each $a \in A$, thus,
$$ \eta(\lambda^{\sigma}) = \overline{\lambda^{\sigma}(\alpha_1)} = (\overline{\lambda(\alpha_1)})^{\sigma} = (\eta(\lambda))^{\sigma}. $$

In order to prove 2), we may notice that $\mathfrak{Y}^{\sigma}$ affords $\phi^{\sigma}=\phi$, thus, $\mathfrak{Y}^{\sigma} = P^{-1} \mathfrak{Y} P$ for some non-singular matrix $P$. We have that
\begin{itemize}
\item $P^{-1} \mathfrak{X} P$ is a projective representation associated with $P^{-1} \mathfrak{Y} P$ and $\alpha$ is the associated factor set;
\item $\mathfrak{X}^{\sigma}$ is a projective representation associated with $\mathfrak{Y}^{\sigma}$ with factor set $\alpha^{\sigma}$.
\end{itemize}
Since $\mathfrak{Y}^{\sigma} = P^{-1} \mathfrak{Y} P$, we have from \cite[Theorem~11.2]{Isaacs} that $\mathfrak{X}^{\sigma} = P^{-1} \mathfrak{X} P \mu$ for some $\mu: G \mapsto \mathbb{C}^{\times}$, which is constant on cosets of $N$ in $G$. It follows that $\alpha^{\sigma} = \alpha \nu$ for $\nu(g,h)=\mu(g)\mu(h)\mu(gh)^{-1} \in B(G,\mathbb{C}^{\times})$; thus, $\beta^{\sigma}=\beta \nu$ and $\overline{\beta^{\sigma}} = \overline{\beta}$ in $H(G,\mathbb{C}^{\times}) = Z(G,\mathbb{C}^{\times})/B(G,\mathbb{C}^{\times})$, as we required.

Finally, we have that
$$ \eta(\lambda^{\sigma}) = (\eta(\lambda))^{\sigma} = (\overline{\beta}^{-1})^{\sigma} = \overline{\beta}^{-1} = \eta(\lambda)  $$
and, thus, $\lambda^{\sigma}=\lambda$.
\end{proof}

\begin{corollary}
\label{corollary:character_triple}
If $(G,N,\phi)$ is a character triple and $\phi$ has values in $K$, for some field $K \geq \mathbb{Q}$, then $(G,N,\phi)$ is strongly congruent to a character triple $(\Gamma, A, \lambda)$ such that $A \leq \ze{\Gamma}$ and $\lambda$ has values in $K$. In particular, $(\Gamma, A, \lambda)$ can be chosen such that it is a canonically constructed character triple isomorphic to $(G, N, \phi)$.
\end{corollary}

\begin{proof}
We can take $(\Gamma, A, \lambda)$ as in Theorem~\ref{theorem:character_triple} and we have that $\lambda$ is $\sigma$-invariant for every $\sigma \in \gal(\cycl{\abs{G}} \mid K)$ and, therefore, $\lambda$ has values in $K$.
\end{proof}

We may notice that, as a consequence of Theorem~\ref{theorem:character_triple}, a character $\psi \in \irr(\Gamma) \smid \lambda)$ has values in $\cycl{n}$, with $n = \abs{G/N} \ell$ and $\ell \in \mathbb{Z}_{> 1}$ is minimal such that $\mathbb{Q}(\phi) \leq \cycl{\ell}$. It follows that the action on $\irr(\Gamma \smid \lambda)$ of a field automorphism $\sigma \in \gal(\cycl{\abs{G}} \mid \mathbb{Q})$, fixing $\phi$, is well defined.

We may also notice that the condition on character field of values of \cite[Theorem~3.5]{Turull:Clifford_theory} now is always verified when $\phi$ is rational-valued. However, we can find a character triple isomorphism which preserves character fields of values even when $\mathbb{Q}(\phi) \neq \mathbb{Q}(\lambda)$, as we will see in Theorem~\ref{theorem:field_in_triples}.

We first need some preliminary results, which are actually nothing more than a ``field preserving" version of some lemmas from \cite[Chapter~11]{Isaacs}.

\begin{lemma}
\label{lemma:epimorphism}
Let $(G,N,\phi)$ be a character triple and let $\mu: G \mapsto \Gamma$ be an onto homomorphism with $\ker(\mu) \leq \ker(\phi)$. Let $M=\mu(N)$, notice that $M \cong N / \ker(\mu)$, and let $\theta$ be the character corresponding to $\phi \in \irr(N/ \ker(\mu))$. Then $(G,N,\phi)$ and $(\Gamma, M, \theta)$ are isomorphic character triples. Moreover, $\mathbb{Q}(\phi) = \mathbb{Q}(\theta)$ and the isomorphism of character triples $\tau$ can be chosen such that it commutes with any $\sigma \in \gal(\cycl{\abs{G}} \mid \mathbb{Q})$ fixing $\phi$.
\end{lemma}

\begin{proof}
Let $\tau : (G,N,\phi) \mapsto (\Gamma, M, \theta)$ be the character triples isomorphism constructed in the proof of \cite[Lemma~11.26]{Isaacs} and notice that, for how it is defined, it commutes with field automorphisms.
\end{proof}

\begin{lemma}
\label{lemma:restriction_isomorphism}
Let $(G,N,\phi)$ be a character triple and let $K \leq \cycl{\abs{G}}$ such that $\mathbb{Q}(\phi) \leq K$. Suppose there exists $\eta \in \irr(G)$ having values in $K$ such that $\eta_N \phi = \theta \in \irr(N)$. Define $\tau: (G,N,\phi) \mapsto (G,N,\theta)$ such that it is the identity map on $G/N$ and such that, for every $N \leq H \leq G$ and every $\psi \in \op{Ch}(H \smid \phi)$, $\tau(\psi)=\psi \eta_H$. Then, $\tau$ is an isomorphism of character triples which commutes with every $\sigma \in \gal( \cycl{ \abs{G} } \mid K )$.
\end{lemma}

\begin{proof}
From \cite[Lemma~11.27]{Isaacs}, we know that $\tau$ is an isomorphism of character triples. Let $\psi \in \op{Ch}(H \smid \phi)$ and let $\sigma \in \gal( \cycl{ \abs{G} } \mid K )$, then $\tau(\psi^{\sigma}) = \psi^{\sigma} \eta_H = (\psi \eta_H)^{\sigma} = (\tau(\psi))^{\sigma}$, because $\eta$ is $\sigma$-invariant.
\end{proof}

\begin{theorem}
\label{theorem:field_in_triples}
Let $(G,N,\phi)$ be a character triple, and let $(\Gamma,A,\lambda)$ be a canonically constructed character triple isomorphic to $(G,N,\phi)$. If $G/N$ is perfect, then the isomorphism of character triples commutes with every $\sigma \in \gal( \cycl{ \abs{G} } \mid \mathbb{Q} )$ fixing $\phi$.
\end{theorem}

We mention that the first published version of this theorem contained a mistake. We thank a reviewer of the previous version for having noticed it.

\begin{proof}
We prove that the isomorphism of character triples constructed in \cite[Theorem~11.28]{Isaacs} commutes with every $\sigma \in \gal( \cycl{ \abs{G} } \mid \mathbb{Q} )$ fixing $\phi$. We actually need to repeat most of the steps of the proof of \cite[Theorem~11.28]{Isaacs}.

As in \cite[Theorem~11.28]{Isaacs}, we define $G^* \leq G \times \Gamma$ as $G^* = \{(g,x) \mid \overline{g}=\pi(x)\}$, where $\overline{g}$ is the image of $g$ in $G/N$ and $\pi$ is the homomorphism of $\Gamma$ onto $G$, and we take $L = N \times A \lhd G^*$. We also define $\phi^*$ and $\lambda^*$ on $L$ by $\phi^*(n,a) = \phi(n)$ and $\lambda^*(n,a)=\lambda(a)$ and we notice that $\phi^*, \lambda^* \in \irr(L)$ and they are invariant in $G^*$. Moreover, we can see from \cite[Problem~10.10]{Isaacs} that $m_{\mathbb{Q}}(\phi^*) = m_{\mathbb{Q}}(\phi)$.

Let $\mu_G$ and $\mu_{\Gamma}$ be the projection homomorphism of $G^*$ onto, respectively, $G$ and $\Gamma$; then $\ker(\mu_G) = 1 \times A \leq \ker(\phi^*)$ and $\ker(\mu_{\Gamma}) = N \times 1 \leq \ker(\lambda^*)$ and it follows from Lemma~\ref{lemma:epimorphism} that there exist character triples isomorphisms between $(G^*,L,\phi^*)$ and $(G,N,\phi)$, and between $(G^*,L,\lambda^*)$ and $(\Gamma,A,\lambda)$, which commute with every $\sigma \in \gal( \cycl{ \abs{G} } \mid \mathbb{Q} )$ fixing $\phi$. We only need to prove that such isomorphism exists also between $(G^*,L,\lambda^*)$ and $(G^*,L,\phi^*)$ and, by Lemma~\ref{lemma:restriction_isomorphism}, we are done when we prove that the character $\phi^*(\lambda^*)^{-1}$ has an extension to $G^*$ having values in $\mathbb{Q}(\phi)$.

Let $\mathfrak{Y}$ be an irreducible $\mathbb{C}$-representation affording $\phi$, and let $\mathfrak{X}$ be a projective $\mathbb{C}$-representation associated with $\mathfrak{Y}$ as in \cite[Theorem~11.2]{Isaacs}. Let $\alpha \in Z(G,\mathbb{C}^{\times})$ be the associated factor set and let $\beta$ be the corresponding factor set of $G/N$ as in \cite[Theorem~11.7]{Isaacs}. Let $\{ x_{\overline{g}} \mid \overline{g} \in G/N \}$ be a set of coset representatives for $A$ in $\Gamma$, such that $\pi(x_{\overline{g}})=\overline{g}$ and $x_{\overline{1}} = 1$, and let $\delta \in Z(G/N,A)$ such that $x_{\overline{g}}x_{\overline{h}} = \delta(\overline{g},\overline{h}) x_{\overline{g}\overline{h}}$, as in \cite[Lemma~11.9]{Isaacs}. Since $\eta(\lambda)={\overline{\beta}}^{-1}$, we have that $\lambda(\delta)\beta \in B(G/N,\mathbb{C}^{\times})$ and hence
$$ \lambda(\delta(\overline{g},\overline{h}))\alpha(g,h) = \nu(\overline{g})^{-1}\nu(\overline{h})^{-1} \nu(\overline{g}\overline{h}) $$
for some function $\nu: G/N \mapsto \mathbb{C}^{\times}$.

Define $\mathfrak{Z}$ on $G^*$ by
$$ \mathfrak{Z}(g,ax_{\overline{g}}) = \mathfrak{X}(G) \lambda(a)^{-1} \nu(\overline{g}). $$
By proceeding as in \cite[Theorem~11.28]{Isaacs} we have that $\mathfrak{Z}$ is a representation of $G^*$ and, if $\psi$ is the character affording $\mathfrak{Z}$, we have that $\psi_L = \phi^*(\lambda^*)^{-1}$.

Since $G^*/L \cong G/N$ is perfect, it follows that $\psi$ is the only extension of $\phi^*(\lambda^*)^{-1}$ to $G^*$, which has values in $\mathbb{Q}(\phi)$ because it is fixed by every $\sigma \in \gal( \cycl{ \abs{G} } \mid \mathbb{Q} )$ fixing $\phi^*(\lambda^*)^{-1}$, and $\mathbb{Q}(\phi^*(\lambda^*)^{-1}) = \mathbb{Q}(\phi)$. The proof is now complete.
\end{proof}

We conclude this section proving a result which we will need for dealing with Schur representation groups for direct products of perfect groups. The following result is a corollary of \cite[V.25.10~Sats]{Huppert:Endliche_Gruppen} and it is probably already known.

\begin{corollary}
\label{corollary:Schur_direct_product}
Let $G$ be a finite group such that $G \cong G_1 \times ... \times G_k$, for some perfect groups $G_i$, and for each $1 \leq i \leq k$ let $\Gamma_i$ be a Schur representation group for $G_i$. Then, $\Gamma = \Gamma_i \times ... \times \Gamma_k$ is a Schur representation group for $G$.
\end{corollary}

\begin{proof}
For each $i$, let $A_i = \ze{\Gamma_i}$, and let $A=A_i \times ... \times A_k = \ze{\Gamma}$; then $\Gamma/A \cong G_1 \times ... \times G_k$ and we have that $\Gamma$ is a central extension for $G$. Since, for every $i$, we have that $A_i < {\Gamma_i}' \leq \Gamma'$, then $A < \Gamma'$ and, by \cite[Theorem~11.19]{Isaacs}, the standard map $\eta: \irr(A) \mapsto M(G)$ is one-to-one. Thus, in order to prove that $\eta$ is an isomorphism, we only need to prove that $\abs{A}=\abs{M(G)}$.

Since, for every $1 \leq i \leq k$, $\Gamma_i$ is a Schur representation group for $G_i$, we have that $\abs{A_i}=\abs{M(G_i)}$. On the other hand, since each $G_i$ is perfect, it follows from \cite[V.25.10~Sats]{Huppert:Endliche_Gruppen} that
$$ M(G) \cong M(G_1) \times ... \times M(G_k) $$
and, thus, $\abs{A}=\abs{M(G)}$ and the thesis follows.
\end{proof}

\section{Non-abelian sections}
\label{section:Non-abelian_sections}

In this section we deal with non-abelian chief factors of a group $G$. As we will see in Section~\ref{section:main}, we only need to study the case when the chef factor is isomorphic to the direct product of some projective special linear groups.

\begin{lemma}
\label{lemma:characters_of_SL}
Let $G \cong \op{SL}(2,p^a)$ for some prime $p$ and some integer $a$ such that $p^a \equiv 3 \; (\op{mod} \: 8)$, and let $r \neq p$ be a prime number. Then there exists a unique faithful character $\chi \in \irr(G)$ of degree $\chi(1)=p^a - 1$ and of values in $\cycl{r}$. In particular, $\chi$ is rational-valued.

Moreover, such character $\chi$ is of minimal degree among the faithful characters having values in $\cycl{r}$.
\end{lemma}

\begin{proof}
We will prove the lemma using the results and the notation of \cite[Theorem~38.1]{Dornhoff:Group_Representation_TheoryA}. In particular, we take $\chi = \theta_k$, for an integer $k$ such that $p^a + 1 = 4k$. Notice that $k$ is well defined, since $p^a \equiv 3 \; (\op{mod} \: 8)$ and, thus, $4 \mid p^a + 1$.

We have by definition that $\chi(1)=p^a - 1$. Moreover, since $8 \nmid p^a + 1$, then $k$ is odd and we can see from \cite[Theorem~38.1]{Dornhoff:Group_Representation_TheoryA} that $\chi$ is faithful. Finally, $\chi$ has values in $\mathbb{Q}(\sigma^k + \sigma^{-k})$, for some primitive $(p^a + 1)$-th root of unity $\sigma$. However, since $p^a + 1 = 4k$ we have that $\sigma^k$ is a 4th root of unity and, thus, $\sigma^k + \sigma^{-k} \in \mathbb{Q}$ and $\chi$ is rational-valued.

We now prove the uniqueness of $\chi$ among the irreducible characters of $G$ of same degree. Let $\psi \in \irr(G)$ such that $\psi(1) = \chi(1)$. Then, by \cite[Theorem~38.1]{Dornhoff:Group_Representation_TheoryA}, $\psi = \theta_j$ for some $1 \leq j \leq (p^a - 1)/2$. By definition, $\psi$ has values in $\mathbb{Q}(\sigma^j + \sigma^{-j})$, and it follows from Niven's theorem that $\mathbb{Q}(\sigma^j + \sigma^{-j}) \leq \cycl{r}$ if and only if $j$ is a multiple of $h = (p^a + 1)/m$, for some $m \in \{4,6,r,2r\}$ dividing $p^a + 1$. However, $\ze{G} = \ker(\theta_j)$ if $j$ is even and, thus, if we require $\psi$ to be faithful, $h$ must be an odd integer. Since $4 \mid p^a + 1$, we must take $m = 4$ and we have that $h=k$. Thus, $j = lk$ for some odd integer $l$ and, since $j \leq (p^a - 1)/2$, we have that $l = 1$, $j = k$ and, therefore, $\psi = \chi$.

Finally, since by \cite[Theorem~38.1]{Dornhoff:Group_Representation_TheoryA} the only non-principal irreducible characters of $G$ having degree less then $p^a - 1$ are non-rational and have values in $\mathbb{Q}_{p}$, the lemma is proved.
\end{proof}

\begin{theorem}
\label{theorem:existing_over_PSL}
Let $G$ be a finite group and let $M,N \lhd G$, $N < M$. Suppose that $M/N = S_1 \times ... \times S_k$ and that, for each $1 \leq i \leq k$, $S_i \cong \op{PSL}(2,p^a)$, for some prime $p$ and some integer $a$ such that $p^a \equiv 3 \; (\op{mod} \: 8)$. Let $r \neq p$ be a prime number and let $\phi \in \irr_{p'}(N)$ be a $G$-invariant character having values in $\cycl{r}$, then there exists a $G$-invariant character $\theta \in \irr_{p'}(M)$ having values in $\cycl{r}$.
\end{theorem}

\begin{proof}
Let $\Gamma = \Gamma_1 \times ... \times \Gamma_k$ such that, for each $1 \leq i \leq k$, $\Gamma_i \cong \op{SL}(2,p^a)$; then by Corollary~\ref{corollary:Schur_direct_product} $\Gamma$ is a Schur representation group for $M/N$. Let $A = \ze{\Gamma}$, then $(\Gamma,A,\lambda)$ is a canonically constructed character triple isomorphic to $(M,N,\phi)$, for some $\lambda \in \lin(A)$. In particular, $A = \ze{G} = A_1 \times ... \times A_k$, with $A_i = \ze{\Gamma_i}$ for each $1 \leq i \leq k$, and we can write $\lambda = \lambda_1 \times ... \times \lambda_k$, with $\lambda_i \in \lin(A_i)$.

For each $1 \leq i \leq k$, take $\psi_i \in \irr(\Gamma_i)$ such that
\begin{itemize}
\item if $\lambda_i = 1_{A_i}$, then $\psi_i = 1_{\Gamma_i}$, and
\item if $\lambda_i \neq 1_{A_i}$, then $\psi_i$ is the only faithful irreducible character of $\Gamma_i$ of degree $\psi_i(1)=p^a - 1$ and of values in $\cycl{r}$, which exists because of Lemma~\ref{lemma:characters_of_SL}.
\end{itemize}
Then, the character $\psi = \psi_1 \times ... \times \psi_k \in \irr(\Gamma)$ lies over $\lambda$ and, in particular, it is the only irreducible constituent of $\lambda^{\Gamma}$ of minimal degree among the ones having values in $\cycl{r}$.

Let $\theta \in \irr(M)$ be the character corresponding to $\psi$ in the character triple isomorphism, then $\theta(1)/\phi(1) = \psi(1) = (p^a - 1)^b$, for some positive integer $b$, because of \cite[Lemma~11.24]{Isaacs}, and it follows that $p \nmid \theta(1)$. Moreover, it follows from Theorem~\ref{theorem:field_in_triples} that $\theta$ has values in $\cycl{r}$.

It only remains to prove that $\theta$ is $G$-invariant. Let $g \in G$ and let $\psi_1 \in \irr(\Gamma \smid \lambda)$ be the character corresponding to $\theta^g \in \irr(M \smid \phi)$. Then, $\psi_1$ has values in $\cycl{r}$ by Theorem~\ref{theorem:field_in_triples}, and $\psi_1(1) = \theta^g(1)/\phi(1) = \psi(1)$. It follows by construction that $\psi_1 = \psi$ and, thus, $\theta^g = \theta$ and the theorem is proved.
\end{proof}

\section{Simple groups}

In this section, we prove Theorem~\ref{result:general} for almost simple groups. Because of \cite[Theorem~C]{Navarro-Tiep:Character_cyclotomic_field}, we actually need to study only groups whose socle is isomorphic to $\op{PSL}(2,3^a)$ for some odd $a > 1$, since otherwise the theorem follows from \cite[Theorem~C]{Navarro-Tiep:Character_cyclotomic_field}, as we will see in Corollary~\ref{corollary:simple_groups}.

In this section we also prove Theorem~\ref{theorem:lie_type} and Corollary~\ref{corollary:simple_coprime}, which will be needed for the proof of Theorem~\ref{result:solvable}. The results we prove here are actually stronger than what we need.

\begin{proposition}
\label{proposition:simple_group_PSL}
Let $r$ be a prime number, let $S \cong \op{PSL}(2,3^a)$, for some odd integer $a > 1$, and let $P$ be a $3$-group acting on $S$ such that $P \cap S \in \syl{3}{S}$. Suppose that $r$ divides $\abs{\no{S}{P}}$. Then, $r$ is odd and there exists a non-principal $P$-invariant character $\phi \in \irr_{3'}(S)$ with values in $\cycl{r}$.
\end{proposition}

\begin{proof}
We first prove that $r$ is odd, since it is very easy to see: let $g \in \no{S}{P}$, then $g$ also normalizes $P \cap S \in \syl{3}{P}$. It follows that $\no{S}{P} \leq \no{S}{P \cap S}$, which has odd order.

To prove the rest, we keep the notation of \cite[Theorem~38.1]{Dornhoff:Group_Representation_TheoryA} for the irreducible characters of $\PSL(2,q)$, where $q=3^a$. The group $S$ admits exactly two irreducible characters $\eta_1,\eta_2$ of degree $\frac{q-1}{2}$, called Weil characters, having values in $\cycl{3}$. Clearly, these characters are fixed by automorphisms of odd order, proving our statement for $r=3$.

Assume $r > 3$. If $P\leq S$, then $r$ divides $\frac{q -1}{2}$, as $\abs{\no{S}{P}} = \frac{q(q-1)}{2}$. Since $i = \frac{q-1}{r}$ is even and $i \leq \frac{q-1}{3}$, we can take as $\phi$ the character $\chi_i$, which is an irreducible character of $S$ having degree $\varphi(1)= q + 1$ and values in $\cycl{r}$.

Now, suppose that $P \not\leq S$. Let $g \in \no{S}{P}$ be an $r$-element and notice that, as we have already seen, $g$ normalizes $P\cap S \in \syl{3}{S}$. This implies that $g=\mathrm{diag}(\lambda,\lambda^{-1})$, for some $\lambda \in \mathbb{F}_q^*$. Let $\alpha \in P$ be a field automorphism of order $3^b>1$, where $a=3^bc$. Then 
$$(\alpha^{-1})^g = g^{-1} \alpha^{-1} g = (g^{-1}g^\alpha)\alpha^{-1}$$
belongs to $P$ if and only if $g^{-1} g^\alpha \in P$. From $g^{-1}g^\alpha=\mathrm{diag}(\lambda^{-1}\lambda^\alpha, \lambda\lambda^{-\alpha})$, we get that $(\alpha^{-1})^g \in P$ if and only if $\lambda^\alpha=\lambda$. Let $\nu \in \mathbb{F}_q^*$ be an element of order $q-1$. Then $\nu^\alpha = \nu^{3^c}$ and it follows that $\lambda \in \langle \nu^k \rangle$, where $k=\frac{3^a-1}{3^c-1}$. Since $g$ is an $r$-element, then $r$ divides $3^c - 1$. So, we can write $3^c - 1 = rh$ for some even integer $h$. Now, take the irreducible character $\chi_{kh}$ of $\SL(2,q)$, having degree $q+1$. Note that $kh=\frac{q-1}{r}\leq\frac{q-3}{2}$ is even, whence $Z(\SL(2,q))\leq \ker(\chi_{kh})$. Moreover, $\mathbb{Q}(\chi_{kh})\leq \mathbb{Q}(\rho^{kh})=\mathbb{Q}_r$.

We show that $\chi_{kh}$ is fixed by $\alpha$: looking at the character table of $\SL(2,q)$, it suffices to consider the conjugacy class of $\xi=\mathrm{diag}(\nu, \nu^{-1})$. Taking a primitive $(q-1)$-root of unity $\rho\in\mathbb{C}$, from
$$(\rho^{kh})^{3^c-1} = \rho^{\frac{(q-1)(3^c-1)}{r}} =(\rho^{q-1})^{h}=1$$
we get $\rho^{kh3^c}=\rho^{kh}$. Hence,
$$\chi_{kh}^\alpha(\xi)=\chi_{kh}(\xi^\alpha)=\chi_{kh}(\xi^{3^c})= \rho^{kh 3^c}+\rho^{-kh 3^c}=\rho^{kh}+\rho^{-kh} = \chi_{kh}(\xi)$$
shows that $\chi_{kh}$ is fixed by $\alpha$, as required. Finally, if we take $\alpha$ such that it generates the cyclic group $PS/S$, we have that the character $\chi_{kh}$ is $P$-invariant, as we required.
\end{proof}

\begin{theorem}
\label{theorem:lie_type}
Let $S$ be a non-abelian simple group of Lie type. Then, there exists a non-principal, rational-valued character $\phi \in \irr(S)$ of odd degree which is invariant under field automorphisms.
\end{theorem}

\begin{proof}
Let the group $S$ be a finite simple group of Lie type in characteristic $\ell$. If $\ell$ is odd, we can take as $\phi$ the Steinberg character of $S$. In fact, this irreducible character is rational-valued and its  degree equals the order of a Sylow $\ell$-subgroup of $S$. Since it is the unique irreducible character of $S$ of $\ell$-defect zero (see \cite[Theorem 4]{Curtis:The_Steinberg_character}), it is $\Aut(S)$-invariant.

So, we may assume $\ell=2$. We make use of the Lusztig's parametrization of the irreducible characters of a group of Lie type (see \cite{Carter:Finite_groups_of_Lie_type,Digne-Michel:Representations_finite_groups}), adapting the ideas of \cite[Lemma 9.2 and Theorem 9.5]{Navarro-Tiep:Rational_irreducible_characters}. Let $\G$ be a simple algebraic group of adjoint type in characteristic $2$ and let $F$ be a Frobenius map on $\G$ such that $S$ is derived subgroup of $G=\G^F$. Let $(\G^\ast, F^\ast)$ be the dual of $(\G,F)$, and write $G^\ast={\G^\ast}^{F^\ast}$.

First, consider the case when $G$ is not a Suzuki group. Then, we can embed $X=\SL_2(q)$ in $G^*$, for some power $q$ of $2$. In particular, we can embed a non-central element $s \in \SL_2(2)\leq X$ of order $3$. This element $s$ is rational and the $G^\ast$-conjugacy class $[s]$ is invariant under field automorphisms of $G^\ast$. Since $Z(\G)=1$, the centralizer of the semisimple element $s\in \G^*$ is connected. By \cite[Lemma 9.1]{Navarro-Tiep:Rational_irreducible_characters} and \cite{Lusztig:Representations_reductive_groups_disconnected} the corresponding semisimple character $\chi_s$ is a non-linear irreducible character of $G$ which is rational-valued, of odd degree and invariant under fields automorphisms of $G$. If $|G:S|$ is a power of $2$, then we can take as $\phi$ the restriction of $\chi_s$ to $S$. Hence, we are left to the groups of type $A_{n-1}^\epsilon(q)$ and $E_6^\epsilon(q)$ with  $q=2^a$, $\epsilon=\pm$ and $(n,q-\epsilon)$ not a $2$-power.

We start considering the groups $S=\PSL_3^\epsilon(q)$ with $\gcd(q- \epsilon,3)=3$. By \cite{Simpson-Frame:The_character_table_for} the group $S$ admits three irreducible rational-valued characters of degree $\frac{(q^2+\epsilon q+1)(q+\epsilon)}{3}$. Looking at the character table in \cite{Simpson-Frame:The_character_table_for}, it is easy to see that one of them is fixed by the field automorphism of order $a$. For the other groups, let $\H$ be a simple simply connected algebraic group and let $F$ be a Frobenius map on $\H$ such that $S=\frac{H}{Z(H)}$, where $H=\H^F$. Let $(\H^\ast,F^\ast)$ be the dual of $(\H,F)$. Suppose $S=\PSL_n^\epsilon(q)$, with $n\geq 5$. The group ${\H^\ast}^{F^\ast}=\PGL_n(q)$ contains a rational non-central element $s$ of order $3$, whose preimage $\hat s$ in $\GL_n^\epsilon(q)$ is $\textrm{diag}\left(\begin{pmatrix} 0 & 1 \\ 1 & 1\end{pmatrix}, \mathrm{I}_{n-2}\right)$. As shown in the proof of \cite[Theorem 9.5]{Navarro-Tiep:Rational_irreducible_characters}, the element $s$ belongs to the derived subgroup of ${\H^\ast}^{F^\ast}$ and its centralizer $C_{\H^\ast}(s)$ is connected. Hence, the corresponding semisimple character $\chi_s$ is a non-principal irreducible rational-valued character of $H$ such that $Z(H)\leq \ker(\chi_s)$. Furthermore, since the conjugacy class $[s]$ is invariant under field automorphisms, $S$ admits a character $\phi$ as in the statement. Now, suppose that $S$ is of type $E_6^\epsilon(q)$. Again by \cite[Theorem 9.5]{Navarro-Tiep:Rational_irreducible_characters}, we can take a rational non-central element $s\in F_4(2)$ of order $5$ that can be embedded in ${\H^\ast}^{F^\ast}$ and whose centralizer $C_{\H^\ast}(s)$ is connected. Furthermore, $s$ belongs to the derived subgroup of ${\H^\ast}^{F^\ast}$ and it is fixed by field automorphisms. Then, the semisimple character $\chi_s$ of $H$ gives rise to a non-principal irreducible character $\phi$ of $S$ which is rational-valued, of odd degree and invariant under field automorphisms.

Finally, suppose $G=\Sz(q)$, $q>2$. Then $G^\ast$ contains a unique conjugacy class $[s]$ of elements of order $5$. Arguing as above, the corresponding semisimple character $\chi_s$ is non-principal, rational-valued, of odd degree and $\Aut(G)$-invariant.
\end{proof}

\begin{corollary}
\label{corollary:simple_coprime}
Let $p$ be an odd prime number, let $S$ be a non-abelian simple group such that $p \nmid \abs{S}$ and let $P$ be a $p$-group acting on $S$. Then, there exists a non-principal, rational-valued $P$-invariant character $\phi \in \irr(S)$ of odd degree.
\end{corollary}

\begin{proof}
By replacing $P$ with $P/\ce{P}{S}$, we may assume that $P \leq \Aut(S)$. Since $p \nmid \abs{S}$, then either $P=1$, and the thesis follows for \cite[Theorem~B]{Navarro-Tiep:Rational_irreducible_characters}, or $S$ is a simple group of Lie type and the elements of $P$ are field automorphisms. Now the thesis follows from Theorem~\ref{theorem:lie_type}.
\end{proof}

We can now collect all the facts we proved in this section in the following corollary.

\begin{corollary}
\label{corollary:simple_groups}
Let $p$ and $r$ be two prime numbers, with $p$ odd, let $S$ be a non-abelian simple group and let $P$ be a $p$-group acting on $S$ such that $P \cap S \in \syl{p}{S}$. Assume also that, if $p=3$ and $S \cong \op{PSL}(2,3^a)$ for some odd integer $a > 1$, then $r$ divides $\abs{\no{S}{P}}$. Then, there exists a non-principal $P$-invariant character $\phi \in \irr_{p'}(S)$ with values in $\cycl{r}$. Moreover, if $p \nmid \abs{S}$, then we can choose $\phi$ such that it has odd degree.
\end{corollary}

\begin{proof}
If $p=3$ and $S \cong \op{PSL}(2,3^a)$, for some odd integer $a > 1$, then it follows from Theorem~\ref{proposition:simple_group_PSL}. Otherwise, let $G=PS$, then it follows from \cite[Theorem~C]{Navarro-Tiep:Character_cyclotomic_field} that there exists a non-trivial, rational-valued character $\chi \in \irr(G)$ of degree not divisible by $p$. Since $G/S$ is a $p$-group, then $\chi_S \in \irr(S)$ and it is rational-valued and $P$-invariant. Finally, the last part of the thesis follows from Corollary~\ref{corollary:simple_coprime}.
\end{proof}

Notice that, in Corollary~\ref{corollary:simple_groups}, it is not always possible to find $\phi \in \irr_{p'}(S)$ of odd degree, if $p \mid \abs{S}$. Examples of this fact are the group $\PSL(3,2)$, for $p=7$, and $\PSL(2,11)$, for $p=11$.

\section{Proof of the main theorems}
\label{section:main}

In this section, we are finally going to prove our main theorems. First, we need a result from \cite{Isaacs} to deal with abelian chief factors.

\begin{proposition}
\label{proposition:abelian_section}
Let $S$ be a solvable group acting on a finite group $G$ and let $N \lhd SG$ such that $(\abs{S},\abs{G:N})=1$ and $\ce{G/N}{S}=1$. If $\phi \in \irr(N)$ is $S$-invariant, there exists a unique $S$-invariant character in $\irr(G \! \mid \! \phi)$.
\end{proposition}

\begin{proof}
This is part of \cite[Problem~13.5]{Isaacs}. Notice that the condition that $\ce{G}{S} \leq N$, which appears in \cite[Problem~13.5]{Isaacs}, is weaker then asking that $\ce{G/N}{S}=1$.
\end{proof}

\begin{theorem}
\label{theorem:existing_over}
Let $G$ be a finite group, let $P$ be a Sylow $p$-subgroup of $G$, for some odd prime number $p$, let $r \neq p$ be a prime number and let $N \lhd G$ such that $(\abs{\no{G/N}{P}},2r)=1$. Suppose that either $G/N$ is solvable or $p=3$ and every non-abelian chief factor of $G/N$ is isomorphic to the direct product of copies of $\op{PSL}(2,3^a)$, for some odd integer $a > 1$.

Let $\phi \in \irr_{p',\cycl{r}}(N)$ be $P$-invariant, then there exist a group $P \leq U \leq G$ and a character $\chi \in \irr_{p',\cycl{r}}(U \! \mid \! \phi)$ such that $\chi^G \in \irr(G)$.

Moreover, suppose that $G/N$ is solvable, then:
\begin{itemize}
\item if $2 \nmid \phi(1)o(\phi)$, then $\chi$ can be chosen such that $2 \nmid \chi(1)o(\chi)$;
\item if $\phi$ is rational-valued and $\phi(1)$ is a power of $2$, then $\chi$ can be chosen such that it is rational-valued and $\chi(1)$ is a power of $2$.
\end{itemize}
\end{theorem}

\begin{proof}
We prove the theorem by induction over $\abs{G:N}$. First, we need to observe that, if $X/N$ is a subgroup of $G/N$, then any chief factor of $X/N$ must be a section of a chief factor of $G/N$. However, it follows from \cite[II.8.27~Hauptsats]{Huppert:Endliche_Gruppen} that any subgroup of $\op{PSL}(2,3^a)$, with $a$ odd, is either solvable or it is isomorphic to the group $\op{PSL}(2,3^b)$, for some $b \mid a$. Thus, the hypothesis we made on chief factors of $G/N$ holds also for the chief factors of any subgroup $X/N \leq G/N$.

Now, let $T = \op{I}_G(\phi)$ and suppose that $T < G$. Since $\phi$ is $P$-invariant, then $P \leq T$ and, by induction, we have that there exists $P \leq U \leq T$ and $\chi \in \irr_{p',\cycl{r}}(U \! \mid \! \phi)$ as required. Finally, by induction $\chi^T \in \irr(T \! \mid \! \phi)$ and, thus, $\chi^G$ is irreducible.

Therefore, we can assume that $\phi$ is $G$-invariant. Let $N < M \lhd G$ such that $M/N$ is a chief factor for $G/N$ and notice that, by Frattini argument, $\no{G/M}{P} = \no{G/N}{P}M/M$ and, therefore, the conditions on the normalizer of $P$ are preserved by the quotient.

First suppose that $M/N$ is an abelian $s$-group, for some prime number $s$. If $s \notin \{r,2\}$, then the character $\phi$ is $s$-rational and it follows from \cite[Theorem~6.30]{Isaacs} that there exists a unique $s$-rational character $\theta \in \irr(M)$ extending $\phi$, which must be $G$-invariant and with values in $\mathbb{Q}(\phi) \leq \cycl{r}$. Suppose now that $s \in \{ 2, r \}$, then we have that
$$ \ce{M/N}{P} = \no{M/N}{P} = 1, $$
since nor $2$ neither $r$ divides $\abs{\no{G/N}{P}}$, and it follows from Proposition~\ref{proposition:abelian_section} that there exists a unique $P$-invariant character $\theta \in \irr(M)$ lying over $\phi$. Since the action of $P$ on $\irr(M)$ commutes with the action of field automorphisms, we have that $\mathbb{Q}(\theta) = \mathbb{Q}(\phi) \leq \cycl{r}$.

Suppose that $2 \nmid \phi(1)o(\phi)$. If $s=2$, we have that $\theta$ is the unique extension of $\phi$ such that $2 \nmid \theta(1)o(\theta)$. If $s \neq 2$, then $\frac{\theta(1)}{\phi(1)}$ is a power of $s$ and $o(\theta) \mid o(\phi) \abs{M/N}$, and it follows that $2 \nmid \theta(1)o(\theta)$ also in this case.

Suppose now that $\mathbb{Q}(\phi) = \mathbb{Q}$ and $\phi(1)$ is a power of $2$. We have already proved that $\mathbb{Q}(\theta) = \mathbb{Q}(\phi) = \mathbb{Q}$. If $s = 2$, then $\frac{\theta(1)}{\phi(1)}$ is a power of $2$; else, by \cite[Theorem~6.30]{Isaacs}, $\theta$ is the only rational-valued irreducible constituent of $\phi^M$ and it extends $\phi$.

If $G/N$ is solvable, then so is $G/M$ and, by induction, there exists $P \leq U \leq G$ and $\chi \in \irr_{p',\cycl{r}}(U)$ with the required properties, with $\chi$ lying over $\theta$ and, thus, over $\phi$.

Suppose now that $G/N$ is not solvable and that $M/N$ is non-abelian. Then, $p=3$ and $M/N$ is isomorphic to the direct product of copies of $\op{PSL}(2,3^a)$, for some odd integer $a > 1$, and it follows from Theorem~\ref{theorem:existing_over_PSL} that there exists a $G$-invariant character $\theta \in \irr_{p'}(M)$ having values in $\cycl{r}$. By induction, again, there exists a group $P \leq U \leq G$ and a character $\chi \in \irr_{p'}(U)$ with values in $\cycl{r}$, lying over $\theta$ and, thus, over $\phi$, and such that $\chi^G \in \irr(G)$. 
\end{proof}

We are approaching our last proof; first, however, few more preliminary results are needed, to prove Theorem~\ref{result:solvable}.

\begin{lemma}
\label{lemma:sigma-invariant}
Let $P$ be a $p$-group, for some odd prime number $p$, and let $\op{id} \neq \sigma \in \op{Aut}(\cycl{p} \mid \mathbb{Q})$. Then, the principal character is the only $\sigma$-invariant irreducible character of $P$.
\end{lemma}

\begin{proof}
Let $\chi \in \irr(P)$ be a $\sigma$-invariant character of $P$ and let $N \leq \ze{P}$ such that $\abs{N}=p$, then $\chi_N = \chi(1) \lambda$ for some linear character $\lambda \in \irr(N)$. Since $\chi$ is $\sigma$-invariant, so is $\lambda$ and it follows that $\mathbb{Q}(\lambda)=\mathbb{Q}$ and, thus, $\lambda = 1_N$. As a consequence, $\chi \in \irr(P/N)$ and, by induction, $\chi = 1_G$.
\end{proof}

\begin{proposition}
\label{proposition:over_1P}
Let $G$ be a finite group, let $p$ be an odd prime number, let $P$ be a Sylow $p$-subgroup for $G$ and let $\chi \in \irr(G)$ be a $p$-rational character. Suppose that there exists some prime number $s$ such that $s \mid p-1$ and $s \nmid \chi(1)$. Then, $ s \nmid [\chi_P,1_P]$ and, in particular, $[\chi_P,1_P] \neq 0$.
\end{proposition}

We recall that a character is $p$-rational if it has values in some cyclotomic extension $\cycl{m}$ with $p \nmid m$.

\begin{proof}
Since $s \mid p - 1$, there exists some field automorphism $\sigma \in \op{Aut}(\cycl{p} \!\mid\! \mathbb{Q})$ of order $s$. Since $\chi$ has values in some cyclotomic field $\cycl{m}$ such that $\cycl{p} \cap \cycl{m} = \mathbb{Q}$, we have that $\chi$ is fixed by $\sigma$ and, thus, $\sigma$ permutes the irreducible constituents of $\chi_P$. Since $o(\sigma) = s$ and $s \nmid \chi(1)$, we have that there exists an irreducible constituent $\lambda$ of $\chi_P$ which is fixed by $\sigma$ and such that $s \nmid [\chi_P,\lambda]$. Finally, it follows from Lemma~\ref{lemma:sigma-invariant} that $\lambda = 1_P$ and we are done. 
\end{proof}

The following result proves Theorem~\ref{result:general} and the remaining direction of Theorem~\ref{result:solvable}.

\begin{theorem}
Let $G$ be a finite group and let $p$ and $r$ be two prime numbers. Let $P \in \syl{p}{G}$; if either $2$ or $r$ divide $\abs{\no{G}{P}}$, there exists a non-principal character $\chi \in \irr_{p'}(G)$ having values in $\cycl{r}$. Moreover, if $p \neq r$, $p$ is an odd prime which is not a Fermat number and $G$ is $p$-solvable, then $\chi$ can be chosen such that $[\chi_P,1_P] \neq 0$.
\end{theorem}

\begin{proof}
If either $p=r$ or $p=2$, then the existence of $\chi$ is guaranteed by \cite[Theorem~A]{Navarro-Tiep:Character_cyclotomic_field}. Thus, in proving this theorem we can assume that $p$ and $r$ are distinct and $p>2$.

Let $N \lhd G$ to be maximal such that, for some $K < N$, $N/K$ is a chief factor for $G$ and either
\begin{itemize}
\item $2$ or $r$ divide $\abs{\no{N/K}{P}}$, or
\item $N/K$ is isomorphic to the direct product of some non-abelian simple group $S$ and, if $p = 3$, $S \ncong \op{PSL}(2,3^a)$ for any odd positive integer $a$.
\end{itemize}
Since, by Frattini argument, $\no{G/N}{P} = \no{G}{P}N/N$, we can see that $(\abs{\no{G/N}{P}},2r)=1$ and, in particular, we have that $N > 1$.

Another consequence of the aforementioned Frattini argument is that we can replace $G$ with $G/K$ and we can assume that $N$ is a minimal normal subgroup of $G$. We claim that there exists some character $\phi \in \irr(N)$ having values in $\cycl{r}$ and such that $p \nmid \phi(1)$. Moreover, if $p \nmid \abs{N}$, we claim that $\phi$ can be chosen such that either $2 \nmid \phi(1)o(\phi)$ or $\phi$ is a linear character of order $2$.

Suppose first that either $2$ or $r$ divide $\abs{\no{N}{P}}$ and $N$ is an elementary abelian $s$-group, for some prime $s \in \{ 2,r \}$. Then, $p \nmid \abs{N}$ and, by Glauberman correspondence, since $\ce{N}{P} = \no{N}{P} > 1$, there exists some character $\phi \in \irr(N)$ which is $P$-invariant. Moreover, since $N$ is elementary abelian, then $\phi$ is linear and with values in $\cycl{s}$ and order $o(\phi)=s$.

Suppose now that $N = S_1 \times ... \times S_k$, with $S_i \cong S$, for each $i = 1,...,k$, for some simple group $S$, and notice that, by how we chose $N$, if $p=3$ and $S \cong \op{PSL}(2,3^a)$ for some odd integer $a > 1$, then there exists an element $x \in \no{N}{P}$ such that $o(x) \in \{ 2,r \}$. If such element exists, then we can write it as
$$ x = x_1 \times ... \times x_k \in S_1 \times ... \times S_k $$
and, since $x \neq 1$, we can assume without loss of generality that $x_1 \neq 1$. If we call $\tilde{S} = S_2 \times ... \times S_k$, we can write $x = x_1 \tilde{x}$, with $\tilde{x} \in \tilde{S}$.

Now, let $Q = \no{P}{S_1}$ and notice that $Q$ also normalizes $\tilde{S}$. Since $Q^x = Q$, because $x$ normalizes both $P$  and $S_1$, we have that $Q^{x_1} = Q^{\tilde{x}^{-1}}$. However, $Q^{\tilde{x}^{-1}} \leq Q \tilde{S}$ and it follows that $Q^{x_1} \leq Q S_1 \cap Q \tilde{S} = Q$ and, thus, $x_1$ normalizes $Q$. Notice that, in this case, it follows from Proposition~\ref{proposition:simple_group_PSL} that $o(x_1)=r$.

Therefore, we have a $p$-group $Q = \no{P}{S_1}$ acting on a simple group $S_1$ and, if $p=3$ and $S_1 \cong \op{PSL}(2,3^a)$ for some odd integer $a > 1$, then $\no{S_1}{Q}$ contains an element of order $r$. It follows from Corollary~\ref{corollary:simple_groups} that there exists a non-principal $Q$-invariant character $\xi_1 \in \irr_{p'}(S_1)$ having values in $\cycl{r}$. Moreover, if $p \nmid \abs{N}$, then also $p \nmid \abs{S_1}$ and $\xi_1$ can be chosen such that $\xi_1(1)$ is odd. Since $o(\xi_1)=1$, because $S_1$ is perfect, in this case we have that $2 \nmid \xi_1(1) o(\xi_1)$.

Now, let $X$ be a set of coset representatives for $Q$ in $P$, with $1 \in X$, and define the character $\phi = \xi_1 \times ... \times \xi_k \in \irr(N)$, with $\xi_i \in \irr(S_i)$ for each $i = 1,...,k$, such that
\begin{itemize}
\item $\xi_i = {\xi_1}^g$ for some $g \in X$ such that ${S_1}^g = S_i$, if such element $g$ exists;
\item $\xi_i = 1_{S_i}$ otherwise.
\end{itemize}
We have that $\phi$ is $P$-invariant, it has values in $\cycl{r}$, $o(\phi)=1$, $p \nmid \phi(1)$ and, if $2 \nmid \xi_1(1)$, then also $2 \nmid \phi(1)$. Thus, we have proved the claim we made in the second paragraph.

Consider now the group $G/N$ and notice that, from how we have chosen $N$, we have that each chief factor of $G/N$ is either abelian or, if $p=3$, it is congruent to the direct product of some copies of $\op{PSL}(2,3^a)$ for some odd integer $a > 1$. It follows from Theorem~\ref{theorem:existing_over} that there exist a group $P \leq U \leq G$ and a character $\theta \in \irr_{p',\cycl{r}}(U \! \mid \! \phi)$ such that $\theta^G \in \irr(G)$. Moreover, if $G$ is $p$-solvable, then $p \nmid \abs{N}$ and $T/N$ is solvable and, therefore, we can take $\theta$ such that either $\theta(1)$ is odd or it is a power of $2$. Thus, if $p$ is odd and it is not a Fermat number, we can always find a prime number $s$ which divides $p - 1$ and not $\theta(1)$, and it follows from Proposition~\ref{proposition:over_1P} that, if $G$ is $p$-solvable, then $[\theta_P,1_P] > 0$.

Finally, let $\theta^G = \chi \in \irr(G)$, then $p \nmid \chi(1)$, because $P < T$, $\chi$ has values in $\cycl{r}$, because so has $\theta$, and $\chi$ lies over $1_P$ whenever $\theta$ does. This proves the theorem. 
\end{proof}


\begin{thebibliography}{99}

\bibitem{Carter:Finite_groups_of_Lie_type}
	R.W. Carter.
	\emph{Finite groups of Lie type. Conjugacy classes and complex characters}.
	Pure and Applied Mathematics. A Wiley-Interscience Publication.
	New York: John Wiley \& Sons, Inc., 1985.

\bibitem{Curtis:The_Steinberg_character}
	C.W. Curtis.
	``The Steinberg character of a finite group with a $( B , N )$-pair''.
	\emph{J. Algebra} 4 (1966), 433--441.
	
\bibitem{Digne-Michel:Representations_finite_groups}
	F. Digne and J. Michel.
	\emph{Representations of finite groups of Lie type}.
	Vol. 21. London Mathematical Society Student Texts.
	New York: Cambridge University Press, 1991.
	
\bibitem{Dornhoff:Group_Representation_TheoryA}
	L. Dornhoff.
	\emph{Group Representation Theory Part A}.
	Vol. 7. Pure and Applied Mathematics.
	New York: Marcel Dekker, Inc., 1971.
	
\bibitem{Giannelli-Hung-ShaefferFry-Vallejo:Characters_of_pi_degree_and_small}
	E. Giannelli, N.N. Hung, M. Schaeffer Fry and C. Vallejo.
	``Characters of $\pi'$-degree and small cyclotomic fields''.
	\emph{Ann. Mat. Pura Appl.} 200 (2021), 1055--1073.
	
\bibitem{Guralnick-Navarro-Tiep:Finite_groups_odd_normalizer}
	R. Guralnick, G. Navarro and P.H. Tiep.
	``Finite groups with odd Sylow normalizers''.
	\emph{Proc. Amer. Math. Soc.} 144.12 (2016), 5129–-5139.

\bibitem{Huppert:Endliche_Gruppen}
	B. Huppert.
	\emph{Endliche gruppen I}.
	Berlin: Springer-Verlag, Inc., 1967.

\bibitem{Isaacs}
	I.M. Isaacs.
	\emph{Character theory of finite groups}.
	New York: Dover Publications, Inc., 1994.

\bibitem{Isaacs:Characters_groups_odd_order}
	I.M. Isaacs.
	``Characters and Hall subgroups of groups of odd order''.
	\emph{J. Algebra} 157 (1993), 548--561.

\bibitem{Isaacs-Malle-Navarro:Real_characters}
	I.M. Isaacs, G. Malle and G. Navarro.
	``Real characters of $p'$-degree''.
	\emph{J. Algebra} 278 (2004), 611--620.

\bibitem{Isaacs-Navarro:Character_table_Sylow_normalizer}
	I.M. Isaacs and G. Navarro.
	``Character tables and Sylow normalizers''.
	\emph{Arch. Math.} 78 (2002), 430--434.

\bibitem{Lusztig:Representations_reductive_groups_disconnected}
	G. Lusztig.
	``On the representations of reductive groups with disconnected centre''.
	\emph{Ast\'erisque} 168.10 (1988), 157--166.

\bibitem{Malle-Navarro:Characterizing_normal_sylow}
	G. Malle and G. Navarro.
	``Characterizing normal Sylow $p$-subgroups by character degrees''.
	\emph{J. Algebra} 370 (2012), 402--406.

\bibitem{Navarro-Tiep:Character_cyclotomic_field}
	G. Navarro and P.H. Tiep.
	``Characters of $p'$-degree with cyclotomic field of values''.
	\emph{Proc. Amer. Math. Soc.} 134.10 (2006), 2833-–2837.

\bibitem{Navarro-Tiep:Rational_irreducible_characters}
	G. Navarro and P.H. Tiep.
	``Rational irreducible characters and rational conjugacy classes in finite groups''.
	\emph{Trans. Amer. Math. Soc.} 360 (2008), 2443-–2465.

\bibitem{Navarro-Tiep-Vallejo:McKay_natural_correspondence}
	G. Navarro, P.H. Tiep and C. Vallejo.
	``McKay natural correspondences on characters''.
	\emph{Algebra \& Number Theory} 8 (2014), 1839--1856.

\bibitem{Simpson-Frame:The_character_table_for}
	W.A. Simpson and J.S. Frame.
	``The character table for $\SL(3,q)$, $\SU(3,q^2)$, $\PSL(3,q)$, $\PSU(3,q^2)$''.
	\emph{Canad. J. Math.} 25.3 (1973), 486--494.

\bibitem{Turull:Clifford_theory}
	A. Turull.
	``Clifford Theory with Schur Indexes''.
	\emph{J. Algebra} 170.2 (1994), 661--677.

\bibitem{Turull:Odd_character_correspondences}
	A. Turull.
	``Odd character correspondences in solvable groups''.
	\emph{J. Algebra} 319.2 (2008), 739--758.

\end{thebibliography}
\end{document}